\documentclass[11pt]{amsart}

\usepackage{amsmath, amssymb, amsthm, amsaddr}
\usepackage{graphicx}
\usepackage[usenames,dvipsnames]{xcolor}
\usepackage[margin=1in]{geometry}
\usepackage{enumerate}
\usepackage[utf8]{inputenc}
\usepackage[toc,page]{appendix}
\usepackage{lineno}
\usepackage{bbm}
\usepackage{hyperref}
\usepackage{mhchem}
\usepackage{siunitx}
\usepackage{comment}
\usepackage{amsaddr}
\usepackage{lipsum}

\definecolor{mygreen}{rgb}{0.1,0.75,0.2}

\newtheorem{thm}{Theorem}[section]

\newtheorem{lem}[thm]{Lemma}

\newtheorem{remark}[thm]{Remark}

\newtheorem{example}[thm]{Example}

\numberwithin{equation}{section}

\DeclareMathOperator{\grad}{grad}
\DeclareMathOperator{\ddiv}{div}
\DeclareMathOperator{\vol}{vol}

\newcommand{\la}{\langle}
\newcommand{\ra}{\rangle}
\newcommand{\pt}{\partial}

\newcommand{\ud}{\,\mathrm{d}}

\newcommand{\sP}{\mathcal{P}}

\newcommand{\sS}{\mathcal{S}}

\newcommand{\vx}{{x}}

\newcommand{\V}{\scriptscriptstyle{\text{h}}}

\newcommand{\cv}{X^{\scriptscriptstyle{\text{h}}}}

\newcommand{\vxv}{\vec{x}_{\mathbf{\ell}}^h}

\title{Fluctuations in Wasserstein dynamics on Graphs}
\author{Yuan Gao, Wuchen Li, Jian-Guo Liu}
\email{gao662@purdue.edu; wuchen@mailbox.sc.edu; jian-guo.liu@duke.edu.}
\address{Department of Mathematics, Purdue University, West Lafayette, IN, 47906;\\ 
Department of Mathematics, University of South Carolina, SC, 29208;\\ 
Department of Mathematics and Department of Physics, Duke University, Durham, NC 27708.}

\thanks{Yuan Gao is partially supported by NSF award under DMS-2204288. Wuchen Li is partially supported by AFOSR YIP award No. FA9550-23-1-0087, NSF DMS-2245097, and NSF RTG:2038080. Jian-Guo Liu is partially supported by NSF award under DMS-2106988.}

\begin{document}

\maketitle

\begin{abstract}
In this paper, we propose a drift-diffusion process on the probability simplex to study stochastic fluctuations in probability spaces. We construct a counting process for linear detailed balanced chemical reactions with finite species such that its thermodynamic limit is a system of ordinary differential equations (ODE) on the probability simplex. This ODE can be formulated as a gradient flow with an Onsager response matrix that induces a Riemannian metric on the probability simplex. After incorporating the induced Riemannian structure, we propose a diffusion approximation of the rescaled counting process for molecular species in the chemical reactions, which leads to Langevin dynamics on the probability simplex with a degenerate Brownian motion constructed from the eigen-decomposition of Onsager's response matrix. The corresponding Fokker-Planck equation on the simplex can be regarded as the usual drift-diffusion equation with the extrinsic representation of the divergence and Laplace-Beltrami operator. The invariant measure is the Gibbs measure, which is the product of the original macroscopic free energy and a volume element. When the drift term vanishes, the Fokker-Planck equation reduces to the heat equation with the Laplace-Beltrami operator, which we refer to as canonical Wasserstein diffusions on graphs. In the case of a two-point probability simplex, the constructed diffusion process is converted to one dimensional Wright-Fisher diffusion process, which leads to a natural boundary condition  ensuring  that the process remains within the probability simplex. 
\end{abstract}

\section{Introduction}

Stochastic processes on probability density spaces have been widely studied in physics \cite{Dean,KK}, population genetics \cite{Burden,Jost,Wright}, mathematics \cite{DH,EM,WD,Strum}, and mean field game communities \cite{BCCD,cardaliaguet2019master,CG, GLL}. One example is the Wright-Fisher process \cite{Fisher,Wright}, which describes the diffusion approximation of stochastic processes for population genetics \cite{EM,Jost}. Another typical example is Dean-Kawasaki dynamics \cite{Dean,KK}, which are stochastic partial differential equations (SPDE) representing the fluctuations of the empirical measure of a large interacting diffusion particle system. This represents a common noise effect at the population density level. Nowadays, it is recognized that the Dean-Kawasaki process is a class of drift-diffusion processes in probability density space endowed with the Wasserstein-2 metric \cite{Konarovskyi}. It has been well studied that individual diffusion (the heat equation) is a gradient flow of free energies in the Wasserstein-2 space \cite{Otto,vil2008}. 
During
 the fluctuation of gradient flows, Dean-Kawasaki dynamics represent a reversible Langevin dynamics in Wasserstein space, also known as the Wasserstein drift-diffusion process \cite{WD,Strum}. More recently, Sturm defined canonical Wasserstein diffusion in multiple spatial dimensions \cite{Strum}.

In this paper, we study fluctuations on the probability simplex. The probability simplex is a natural probability space $\mathcal{P}(V)$ when describing the distribution function for a Markov jump process on a finite state space $V$ or on an undirected graph $G=(V,E)$. On probability simplex, we can conduct   explicit
calculations towards understanding the nature of Wasserstein diffusion. In \cite{EM1,M}, the authors showed that the Kolmogorov forward equation of detailed-balanced Markov chains is a gradient flow of free energies in terms of discrete Wasserstein distance on the probability simplex \cite{chow2012,EM1,M}. The natural questions are: \textit{What is the Langevin dynamics on the probability simplex $\mathcal{P}(V)$ beyond the deterministic forward equation with gradient flow structure? What process induces the fluctuations, and how can these fluctuations be studied?}

To answer these questions, we construct a drift-diffusion process on the probability simplex incorporating both the geometric structure suggested by Onsager's response matrix in the gradient flow formulation and the macroscopic free energy. With the constructed diffusion process on the probability simplex $\mathcal{P}(V)$, the corresponding Fokker-Planck equation on $\mathcal{P}(V)$ will be studied. For explicit examples with two-point state spaces $V$, fundamental solutions and the relation with the Wright-Fisher diffusion will be established. 

In Section \ref{sec2}, to study the origin of fluctuations on the probability simplex, we first construct counting process for linear chemical reactions in a container with fixed volume $N$. Here, the reaction rate from molecular species $i$ to molecular species $j$ is given by  $Q_{ij}$, a component in a  $Q$-matrix. The counting process for the molecular numbers, rescaled by $h=\frac{1}{N}$, is called the frequency process $\cv_t$. It can be represented as a random time-changed Poisson process that takes values on the discrete probability simplex $\mathcal{P}_h(V)$. Here $V$ is the collection of molecular species involved in the linear chemical reactions. In the thermodynamic limit $h \to 0$, the mean-field limit equation is a system of ordinary differential equations (ODE)  on the probability simplex $\mathcal{P}(V)$, which can also be regarded as the Kolmogorov forward equation of a continuous-time Markov chain with generator $Q$. Under the detailed balance assumption, this ODE can be symmetrized  by using a convex function as a gradient flow on $\mathcal{P}(V)$ with an Onsager response matrix $K$. More importantly, this response matrix $K$ induces a Riemannian metric $g$ on $\mathcal{P}(V)$, which can be interpreted as the discrete Wasserstein distance on the graph.
To introduce an appropriate diffusion approximation for $\cv_t$, we then study detailed properties of the induced Riemannian structure on $(\mathcal{P}(V), g)$, including the coordinate representation of the metric $g$, the volume form, and the extrinsic representation for the Laplace-Beltrami and Dirichlet form on the probability simplex.

In Section \ref{sec3}, we focus on constructing a diffusion approximation for the process $\cv_t$, leading to Langevin dynamics $X_t$ on the probability simplex $\mathcal{P}(V)$. This construction relies on a degenerate Brownian motion in $\mathbb{R}^{d-1}$ with diffusion coefficient constructed from the eigen-decomposition of Onsager's response matrix $K$. This multiplicative noise is introduced in the backward It\^o sense to ensure the invariant measure is the Gibbs measure $\pi_g = e^{-\frac{\psi^{ss}}{h}}\sqrt{|g|}$ with the original energy potential $\psi^{ss}$ and a volume element $\sqrt{|g|}$; see process $X_t$ in \eqref{528}. The corresponding Fokker-Planck equation on $\mathcal{P}(V)$ is exactly the drift-diffusion equation on $(\mathcal{P}(V), g)$ with the extrinsic representation of divergence and gradient; see \eqref{FP-rho}. Meanwhile, the generator of process $X_t$ is exactly the first variation of the Dirichlet form in the $L^2$ space weighted by the Gibbs measure $\pi_g$. When the potential function is taken to be zero, the Fokker-Planck equation reduces to the pure diffusion equation with the Laplace-Beltrami operator. We refer to this case as canonical Wasserstein diffusion \cite{WD}.
Our construction of the Langevin dynamics ensures the conservation of total mass. However, as with most degenerate diffusion processes on bounded domains, additional boundary conditions are needed to ensure the process stays within the bounded domain.

In Section \ref{sec4}, we study an example on a two-point state space $V=\{1,2\}$. We find a change of variable to convert the process $X_t$ to a standard Wright-Fisher process $Y_t$. Since the coefficients in the $Y_t$-process satisfy Feller's regular condition for the degenerate stochastic differential equation (SDE), we impose the zero-flux boundary condition for the Fokker-Planck equation to ensure the density function is mass-conserving. The fundamental solution of the canonical Wasserstein diffusion processes is computed in the zero-potential case.

The fluctuations of detailed-balanced stochastic processes in continuous states, known as Dean-Kawasaki dynamics \cite{Chavanis,Dean,KK}, have been widely studied in the literature. These works constructed an SPDE in terms of the empirical measure of a large interacting diffusion particle system. Mathematically, Dean-Kawasaki dynamics has been studied as a diffusion process on a probability space supported on a one-dimensional cycle \cite{WD} or in multiple dimensions \cite{Strum}.
 In \cite{WD,Strum}, the Dirichlet form in Wasserstein spaces has been introduced and used to compute the generator of the diffusion process. Meanwhile, \cite{WL} proposes a class of Wasserstein diffusions defined on a probability simplex set embedded with Wasserstein-2 metrics on discrete states, including finite state canonical Wasserstein diffusions. \cite{GL2022,MM,Renger} also present a class of diffusion approximations for general detailed balanced chemical reactions based on large deviations and Wasserstein-2-type metrics defined on a simplex set.

Compared to previous studies, this paper investigates the fluctuations of stochastic processes on finite states, coming from the diffusion approximation of the concentration  of molecular species in the linear chemical reactions in a large volume limit. We construct a drift-diffusion process on the probability simplex incorporating the geometric structure suggested by Onsager's response matrix in the gradient flow formulation. We  present a canonical Dirichlet form on the discrete Wasserstein-2 space. We also derive the diffusion approximation for general detailed balanced chemical reactions. This approach differs from the commonly used chemical Langevin equation \cite{Gill}; see details in Remark \ref{rem3.3}.

As observed in \cite{Renger}, the proposed stochastic process shares similarities in the fluctuation of the hydrodynamics of Ginzburg-Landau models \cite{Spohn,Yau}, where an anti-symmetric Brownian motion is defined on the edge set with a constant diffusion matrix. However, this paper studies a variable diffusion coefficients, constructed from Onsager's response matrix, associated with the designed drift vector fields depending on the Riemannian structure of the probability simplex.

It is  worth mentioning that stochastic processes on a population state are crucial in modeling and applications. Typical examples are stochastic differential equations or diffusions introduced by Wright \cite{Wright} and Fisher \cite{Fisher} in population genetics \cite{Feller50, EM,Jost}, which are critical in modeling the time evolution of allele frequencies with the interaction of biological/population behaviors. In its simplest version (a two-point state space) \cite{EWF},   the diffusion approximation of neutral Wright-Fisher model  studies the evolution of the relative frequencies of two alleles. In this paper, we build an interesting transformation between Wasserstein diffusion and Wright-Fisher diffusion on a two-point state space.

The organization of this paper is as follows. In Section \ref{sec2}, we construct a linear reaction process and review Onsager's gradient flows under detailed balance conditions. Then we study the Riemannian structure of the probability simplex endowed with the Riemannian metric $g$ induced by Onsager's response matrix. In Section \ref{sec3}, we construct a Wasserstein diffusion on the probability simplex and study the corresponding Fokker-Planck equation, which incorporates the geometric structure and a weighted Gibbs measure. In Section \ref{sec4}, we present an example of the Wasserstein drift diffusion process on a two-point space, establishing the transformation to the corresponding Wright-Fisher process and demonstrating the closed-form solutions of the corresponding Fokker-Planck equation with Feller's zero-flux boundary condition for the canonical Wasserstein diffusion process.


\section{Counting Process, Thermodynamic Limit and gradient flow with Riemannian structure on probability simplex}\label{sec2}

Before investigating the fluctuations on the probability simplex, we prepare some preliminary results on the stochastic process construction, gradient flow structure, and Riemannian structure. In Section \ref{sec2.1}, we first construct a stochastic process such that the thermodynamic limit (mean-field limit) is a linear ODE described by the $Q$-matrix. This ODE can also be regarded as a Kolmogorov forward equation for a continuous-time Markov chain on finite states with generator $Q$. In Section \ref{sec2.3}, under the detailed balance condition, we derive the gradient flow structure with Onsager's response matrix $K$. This response operator suggests the definition of a Riemannian metric on the probability simplex. We then study details for this Riemannian structure, including volume element, Laplace-Beltrami operator, and Dirichlet form on the probability simplex; see Sections \ref{sec2.4} and \ref{sec2.5}.

\subsection{Counting processes for linear chemical reactions in large volume}\label{sec2.1}

In this subsection, we construct a stochastic process via a random time-changed Poisson process for linear chemical reactions in a container with a total molecular number $N$. The resulting mean-field limit is a linear ODE described by the $Q$-matrix, which can also be regarded as a Kolmogorov forward equation for a continuous-time Markov chain on finite states with generator $Q$.

Let $Q$ be a $Q$-matrix satisfying the row sum zero 
\begin{equation*}
\sum_{j=1}^d Q_{ij} = 0, \qquad Q_{ij} \geq 0, \quad \text{for } j \neq i.
\end{equation*}
Take $Q$ as the reaction rate for the following  linear reversible chemical reactions 
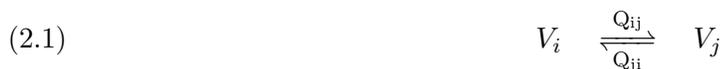
\begin{equation}
V_i 
\quad \ce{<=>[Q_{ij}][Q_{ji}]} \quad
V_j
\end{equation}
with a total of $d$ molecular species ${V} = (V_{i})_{i=1, \cdots, d}$.

Now we describe the above chemical reaction process in terms of the molecular numbers $({C}_t)_{t \ge 0}$ as a counting process. Define the reaction vector for the reaction from $i$ to $j$ as $\vec{\nu}_{ij}=e_j - e_i$, where $e_i$ is the coordinate vector. Based on the law of mass action, the reaction rate for the reaction from $i$ to $j$ is $\Phi_{ij}({c})=Q_{ij} c_i.$
Then this counting process for molecular numbers $({C}_t)_{t \ge 0}$ is constructed via a random time-changed Poisson process
\begin{equation}  
\begin{aligned}
{C}_t = {C}_0 + \sum_{i,j=1, j \neq i}^d  \vec{\nu}_{ij} Y_{ij} \left( \int_0^t  \Phi_{ij}({C_s}) \ud s \right),
\end{aligned}
\end{equation}
where $Y_{ij}(t)$ are i.i.d. unit rate Poisson processes. 
For the linear chemical reactions, the mass conservation and the number of molecular conservation are the same. Indeed, for the mass vector $\vec{m}=(1,1,\cdots, 1)^T$, we have $\vec{\nu}_{ij} \cdot \vec{m}=0$. Thus we have 
\begin{equation}
    \sum_i C_i(t) = \sum_i C_i(0) =: N.
\end{equation}
Denote $h=\frac{1}{N}$, then the rescaled molecular number $\cv_t = h  {C}_t$ is on a discrete probability simplex $\sP_h(V)$ with grid size $h$
$$\cv_t \in \sP_h(V):= \{x^h \in (h\mathbb N)^d;  x_i^h \geq 0,\,\,  \sum_i x_i^h=1\}.$$
We denote $\sS_h=\sP_h(V)$.

The rescaled process satisfies
\begin{equation} \label{Csde}
\begin{aligned}
\cv_t = \cv_0 + \sum_{i,j=1, j \neq i}^d  h\vec{\nu}_{ij} Y_{ij} \left( \int_0^t \frac{1}{h} Q_{ij} X_i^h(s) \ud s \right).
\end{aligned}
\end{equation}
$(\cv_t)_{t \ge 0}$ is a $\sS_h$-valued stochastic process and will be referred to as the frequency process. Indeed one can see whenever $\cv_t$ is on the boundary of $\sS_h$, the reaction direction only points to the interior of $\sS_h$. Thus to ensure $\cv_t$ is a $\sS_h$-valued stochastic process, there is no additional boundary condition needed for linear reactions. However, for general reactions, additional no-reaction boundary condition needs to be imposed \cite{GL2022}. 

\subsection{Thermodynamic limit}
 
The thermodynamic limit (mean-field limit) as $h \to 0$ of process $\cv_t$ was proved by Kurtz \cite{kurtz1971}.
After taking $h \to 0$, we denote the macroscopic concentration variable as $x_t:= \lim_{h \to 0} \cv_t.$

$(x_t)_{t \ge 0}$ are in the probability simplex $\sP(V)$ defined as
$$x_t \in \sP(V):= \{x \in (0,1)^d;  x_i \geq 0,\,\,  \sum_i x_i=1\}.$$
We denote the probability simplex as $\sS:=\sP(V)$.
The mean-field limit $x_t$ satisfies an ODE for $x_t \in \sS$
\begin{equation}\label{RRE}
\dot{x}_i(t) = \sum_{j} Q_{ji} x_j(t).    
\end{equation} 

\subsection{Symmetrization via detailed balance and Onsager's response matrix $K$}\label{sec2.3}
Define the chemical reaction network as a directed weighted graph 
\begin{align}
    G=(V, E, Q), \quad V=\{V_1, V_2, \cdots, V_d\},\\
    Q=(Q_{ij})_{d \times d}, \quad  E=\{(i,j); \, i \neq j, \, Q_{ij}>0\}.
\end{align}

Since the graph is connected, there exists a unique positive eigenvector ${x}^s$ for $Q^\intercal$ such that $$Q^\intercal  {x}^s=0.$$
Assume the detailed balance condition holds
\begin{equation}\label{DB}
    Q_{ij} x^s_i = Q_{ji} x^s_j, \quad i,j = 1, \cdots, d.
\end{equation}
Then the reaction rate equation \eqref{RRE} is called chemical detailed balance.  

Denote the weight as
\begin{equation}
\omega_{ij}:= Q_{ij} x^s_i = \omega_{ji}.
\end{equation}
We still have the following property
\begin{equation}
\begin{aligned}
\omega_{ij} \geq 0, \,\,  \text{for } j \neq i, \quad \sum_j \omega_{ij} = 0,\,\, \sum_j \omega_{ji} = 0 \,\,\,  \text{for } i=1, \cdots, n.
\end{aligned}
\end{equation}
One has directly that $\omega$ is a nonpositive-definite matrix $\omega \in \mathbb{R}^{n \times n}$ satisfying
\begin{equation*} 
\sum_{j=1}^d \omega_{ij}=0, \quad  \omega_{ij}=\omega_{ji}, \quad \sum_{i,j=1}^n \xi_i \omega_{ij} \xi_j = -\frac{1}{2} \sum_{i \neq j} \omega_{ij} (\xi_i - \xi_j)^2 \leq 0 \quad \forall (\xi_i)_{i=1}^n \in \mathbb{R}^n.  
\end{equation*}

Using the detailed balance relation \eqref{DB} and symmetric $\omega$, equation \eqref{RRE} can be recast in a symmetric form
\begin{equation}\label{a}
   \frac{d x_i}{dt} = \sum_{j=1}^d \omega_{ij} \left(\frac{x_j}{x^s_j} - \frac{x_i}{x^s_i}\right) = \sum_{j=1}^d \omega_{ij} \frac{x_j}{x^s_j}. 
\end{equation}

\subsubsection{Variational Structure and Onsager Response Matrix $K$}
Below, we rewrite equation \eqref{a} as an Onsager's gradient flow. For any convex function $\phi(x)$, $\phi''> 0$, we have
\begin{equation}\label{ma1}
  \frac{\ud  x_i}{dt} =  \sum_{j=1}^d \omega_{ij} \theta_{ij}(x) \left[\phi'\left(\frac{x_j}{x^s_j}\right) - \phi'\left(\frac{x_i}{x^s_i}\right)\right],   
\end{equation}
where 
\begin{equation*}
    \theta_{ij}( x)=\theta\left(\frac{x_i}{x^s_i}, \frac{x_j}{x^s_j}\right), \quad \text{with} \quad \theta(s, t) := \frac{s-t}{\phi'(s)- \phi'(t)}.
\end{equation*}
We now recast the above Kolmogorov forward equation as Onsager's gradient flow form.
Define
\begin{equation}\label{Kmatrix}
 K_{ij}( x):= -\omega_{ij} \theta_{ij}( x), \,\, j \neq i, \quad K_{ii}( x) := -\sum_{j=1, j \neq i}^n K_{ij}( x).
\end{equation}
Notice $K(x):=(K_{ij}( x)): \sP(\mathcal{X}) \to \mathbb{R}^{n \times n}$
is a nonnegative definite matrix  satisfying
\begin{equation*}
 \sum_{j=1}^d K_{ij}( x)=0. 
\end{equation*}
Define free energy, also named $\phi$--divergence, on the finite probability spaces 
\begin{equation*}
 \psi^{ss}(x):=\mathrm{D}_{\phi}( x\| x^s) = \sum_{i=1}^n \phi\left(\frac{x_i}{x^s_i}\right) x^s_i.
\end{equation*} 
Then \eqref{ma1} can be recast as an Onsager's gradient flow
\begin{equation} \label{JKO0}
 \frac{\ud x_t}{\ud t} = - K( x_t) \nabla \psi^{ss}(x_t), 
\end{equation}
where $ - \nabla \psi^{ss}$ is the generalized force and $K$ is the Onsager's response matrix.
In terms of component wise, for $i=1,\cdots, d$, 
\[
 \frac{\ud  x_i}{\ud t} = - \sum_{j=1}^d K_{ij}( x) \phi'\left(\frac{x_j}{x^s_j}\right). 
\]
The energy dissipation law can be derived directly 
\begin{equation}
 \frac{\ud }{\ud t} \mathrm{D}_{\phi}( x_t\| x^s) = -  \sum_{i,j=1}^{d} K_{ij}(x) \phi'\left(\frac{x_i}{x^s_i}\right)\phi'\left(\frac{x_j}{x^s_j}\right) \leq 0. 
\end{equation}

\begin{example}[\cite{EM1,M}]  
Kullback--Leibler (KL) divergence corresponds to a particular choice $\phi(x)=x \log x-x+1$, $x \in \mathbb{R}^1$. 
Then the free energy function forms the KL divergence 
\begin{equation*}
 \psi^{ss}(\vx)=\mathrm{D}_{\mathrm{KL}}( x\| x^s)=\sum_{i=1}^n x_i \log \frac{x_i}{x^s_i} - x_i + x_i^s. 
\end{equation*}
and $\theta$ becomes the logarithmic mean function 
\begin{equation}\label{log}
\theta(s,t)=\frac{s-t}{\log s-\log t}, \quad \text{for any } s,t \in \mathbb{R}^1. 
\end{equation}

With this special choice of relative entropy and $\theta$, the $K$-matrix is consistent with the dissipative-conservative decomposition in
\cite[Theorem 3.2, Proposition 4.4]{GL2022}.
We remark that our analysis does not require this special choice of $\phi$--divergence. However, if one chooses this special energy, it represents the global energy landscape of chemical reactions. For detailed balanced nonlinear chemical reactions, the above Onsager's gradient flow still holds with a new matrix $K$; see Remark \ref{rem3.3}.
\end{example}

We  introduce a Riemannian metric on $\sS:=\sP(V)$ so that the above strong form of gradient flow can be recast as a weak form in terms of this metric
\begin{equation}
     g_{x}(\dot{{x}}_t, v) = - \langle \nabla \psi^{ss}({x}_t), v \rangle_{\mathbb{R}^d}, \quad \forall v \in \mathbb{R}^d \text{ with } \sum_{i=1}^d v_i = 0.
\end{equation}
Here the Riemannian metric $g_{x}: T_{x}{\sS} \times T_{x}{\sS} \to \mathbb{R}$ is defined as
\begin{equation}\label{g_d}
    g_{x}(u,v):= \langle u, K^{\dagger}(x) v \rangle_{\mathbb{R}^d}, \quad \forall u,v \in T_{x}{\sS},
\end{equation}
where $K^{\dagger}$ is the pseudoinverse of $K$.
We will see in Section \ref{sec3.4} the coordinate representation and volume element of this $g$.

\subsection{A Riemannian structure for the probability simplex}\label{sec2.4}
In this section, we present some lemmas and properties for the Riemannian metric $g$ introduced by the Onsager response matrix $K$ in \eqref{g_d}. It is known that $K(x)$  is degenerate in most cases for $x\in \pt \sS$, such as when $\theta$ is the logarithmic mean function \eqref{log}. Below, we study the Riemannian structure only for interior points  $x$ in the simplex.

Recall that $K$ is non-negative symmetric, and the sum of the rows is zero. Denote $(d-1)$ orthonormal eigenvectors $u^\ell(\vx) \in \mathbb{R}^d$ of matrix $K$ with positive eigenvalues as 
\begin{equation}\label{Eig}
K(\vx)u^\ell(\vx)=\lambda_\ell(\vx)u^\ell(\vx)  \quad \text{for $\ell=1,2,\cdots, d-1$.}
\end{equation}
$K$ also has an eigenvector
$e=\frac{1}{\sqrt{d}} (1, \cdots , 1)^T$ corresponding to the zero eigenvalue.
Thus $u^\ell$ are orthogonal to $e$. 
These eigenvectors form a $d \times d$ orthogonal matrix
\begin{equation}\label{Q}
  Q = ( u^1|u^2|\cdots|u^{d-1} | e). 
\end{equation}
Write \eqref{Eig} and orthogonality in component-wise, we have for $i=1,\cdots,d$,
\begin{equation}\label{component-wise}
\sum_{i,j=1}^d u_i^k K_{ij} u_j^\ell = \lambda_\ell \delta_{k\ell}, \quad \text{for $k, \ell=1,2,\cdots, d-1$.}
\end{equation}
Since $u^\ell$ is orthogonal to $e$, we know that
$\sum_{i=1}^d u^\ell_i =0$.
The pseudo-inverse $K^{\dagger}$ has a spectral decomposition
\begin{equation}\label{spectral}
K^\dagger = \sum_{\ell=1}^{d-1} \lambda_\ell^{-1} u^\ell \otimes u^\ell.
\end{equation}
Our manifold is the simplex set $\mathcal S = \{ (x^1, \cdots, x^{d-1}, x^d): 0 \leq  x^i \leq 1, i=1, \cdots, d; \sum_{i=1}^d x^i=1 \}$.
We parameterize $\mathcal S$ by $(x^1, \cdots, x^{d-1})$ as follows. Denote domain
$$\Omega_{d-1}:= \{0 < x_i < 1, i = 1, \cdots, d-1; 0 < \sum_{i=1}^{d-1} x_i < 1\} \subset \mathbb{R}^{d-1}.$$ 
Denote the coordinate map as $\varphi$, satisfying
\[
\varphi^{-1}: \Omega_{d-1}  \to \mathcal S, \quad (x^1, \cdots, x^{d-1}) \mapsto \vx = (x^1, \cdots, x^{d-1}, 1 - x^1 - \cdots - x^{d-1}). 
\]
Denote
\begin{equation}\label{Fnot}
F(x^1, \cdots, x^{d-1}) := (f \circ \varphi^{-1})(x^1, \cdots, x^{d-1}) = f(x^1, \cdots, x^{d-1}, 1 - x^1 - \cdots - x^{d-1}).
\end{equation}
For $i=1, \cdots, d-1$, the coordinate derivative can be represented as an extrinsic direction derivative in $\mathbb{R}^d$,   
\[
\frac{\partial}{\partial x^i} F = \tau_i \cdot \nabla f, \quad \tau_i = e_i - e_d.
\]
This gives a basis for the tangent plane $\left\{\frac{\partial}{\partial x^i}\right\}_{i=1}^{d-1}$.

\begin{lem}\label{lem2.2}
Let the Riemannian metric $g$ be defined in \eqref{g_d}. In terms of coordinates $x^1, \cdots, x^{d-1}$ in $\Omega_{d-1}$, we have the following representation for $(g_{ij}), (g^{ij})$, $i,j=1 \cdots, d-1$, and properties for $\det{g}$ and the volume form:
\begin{enumerate}[(i)]
    \item The Riemannian metric is positive-definite and $(g_{ij})_{i,j=1}^{d-1}$ has the following decomposition
\begin{equation}\label{g_metric}
  (g_{ij})_{i,j=1}^{d-1} = Q_{d-1} \Lambda_{d-1}^{-1} Q_{d-1}^\intercal, 
\end{equation}
where $\Lambda_{d-1}$ is a diagonal matrix consisting of positive eigenvalues of $K$ and $Q_{d-1}$ is a $(d-1) \times (d-1)$ matrix 
\[
Q_{d-1} = (q_{\ell, j})_{\ell, j=1}^{d-1}, \quad 
q_{\ell, j} = u_j^\ell - u_d^\ell,
\]
and $(g^{ij}) = K_{d-1}:= (K_{ij})_{i,j=1}^{d-1}$ are the first $(d-1) \times (d-1)$ block of $K$.
    \item $|g| := \text{det } g = \frac{d}{\lambda_1 \cdots \lambda_{d-1}}.$
    \item Define 1-form 
\begin{equation}\label{1form}
\omega^\ell = \sum_{i=1}^{d-1} \omega^{\ell}_i d x^i, \quad \omega^{\ell}_i = \frac{1}{\sqrt{\lambda_\ell}} (u^\ell_i - u^\ell_d).
\end{equation}
The volume form is given by 
\[
\text{vol}_g := \sqrt{|g|} \ud x^1 \wedge \cdots  \wedge \ud x^{d-1}
= \omega^1  \wedge \cdots  \wedge \omega^{d-1} = \frac{1}{\sqrt{\lambda_1 \cdots \lambda_{d-1}}} \ud \mathcal{H}^{d-1}(x),
\]
where $\ud \mathcal{H}^{d-1}(x)$ is the Hausdorff surface measure.
\end{enumerate}
\end{lem}

\begin{proof}
First, from the spectral decomposition \eqref{spectral} of $K^{\dagger}$ we have 
for any $a,b \in T_{\vx} \sS \subset \mathbb{R}^d$, $\sum_{i=1}^d a^i = \sum_{i=1}^d b^i=0$ and
\begin{align*}
 g(a,b) = & \langle a, K^\dagger b \rangle_{\mathbb{R}^d} = \sum_{i,j=1}^d a^i K_{ij}^\dagger b^j = \sum_{\ell=1}^{d-1} \lambda_\ell^{-1} \left(\sum_{i=1}^d a^i u^\ell_i \right) \left(\sum_{j=1}^d b^j u^\ell_j \right)\\
 =& \sum_{\ell=1}^{d-1} \lambda_\ell^{-1} \left(\sum_{i=1}^{d-1} a^i (u^\ell_i - u^\ell_d) \right) \left(\sum_{j=1}^{d-1} b^j (u^\ell_j - u^\ell_d) \right) \\
 =& \sum_{i,j=1}^{d-1} a^i b^j \sum_{\ell=1}^{d-1} (u^\ell_i - u^\ell_d) \lambda_\ell^{-1} (u^\ell_j - u^\ell_d)  \\
 =& \sum_{i,j=1}^{d-1} a^i b^j g_{ij}.
\end{align*}
Thus $(g_{ij})_{i,j=1}^{d-1} = Q_{d-1} \Lambda_{d-1}^{-1} Q_{d-1}^\intercal$ is the representation of $g$ in coordinates $x^1, \cdots, x^{d-1}$.

Next, from \eqref{component-wise} and the row sum zero and symmetry of $K$, we have
\begin{equation*}
\sum_{i,j=1}^d (u_i^k - u_d^k) K_{ij} (u_j^\ell - u_d^\ell) = \lambda_\ell \delta_{k\ell}, \quad 
\text{ and } \,\, 
\sum_{i,j=1}^{d-1} (u_i^k - u_d^k) K_{ij} (u_j^\ell - u_d^\ell) = \lambda_\ell \delta_{k\ell}. 
\end{equation*}
Then the above decomposition is recast as
\begin{equation}\label{tm11}
 Q_{d-1}^\intercal K_{d-1} Q_{d-1} = \Lambda_{d-1} : = \text{diag} (\lambda_1, \cdots, \lambda_{d-1}).
\end{equation}
Taking the inverse, one has 
\begin{equation}
   K_{d-1}^{-1} = Q_{d-1} \Lambda_{d-1}^{-1} Q_{d-1}^\intercal = (g_{ij})_{i,j=1}^{d-1}.  
\end{equation}
Hence $(g^{ij})_{i,j=1}^{d-1} = K_{d-1}$. 
Thus (i) is proved.

Second, we compute the determinant of $Q_{d-1}$ and thus $|g|$. 
Notice \eqref{g_metric} implies
\[
\text{det } g = \frac{(\text{det } Q_{d-1})^2}{\lambda_1 \cdots \lambda_{d-1}}.
\]
To compute $\text{det } Q_{d-1}$, we subtract the last row from each other row of $Q$ in \eqref{Q} to obtain the following matrix, whose determinant remains the same
\[
\begin{pmatrix}
    u^1_1 - u^1_d & \cdots & u^{d-1}_1 - u^{d-1}_d & 0 \\
    u^1_2 - u^1_d & \cdots & u^{d-1}_2 - u^{d-1}_d & 0 \\  
    \cdots \\
    u^1_{d-1} - u^1_d & \cdots & u^{d-1}_{d-1} - u^{d-1}_d & 0 \\
    u^1_d & \cdots & u^{d-1}_d & 1 / \sqrt{d} \\
\end{pmatrix}
=
\begin{pmatrix}
  Q_{d-1} & {0} \\
  u^1_d, \cdots, u^{d-1}_d & 1 / \sqrt{d}
\end{pmatrix}.
\]
Taking the determinant, we obtain
\[
1 = \text{det } Q = \frac{1}{\sqrt{d}} \text{det } Q_{d-1}.
\]
Hence we have
\[
\text{det } g = \frac{d}{\lambda_1 \cdots \lambda_{d-1}}.
\]
Therefore, conclusion (ii) is proved.

Third, we compute the volume form.
Recall the 1-form defined in \eqref{1form}
and a $(d-1)$ form
\[
\omega = \omega^1 \wedge \cdots \wedge \omega^{d-1}.
\]
We show that $\omega$ is the volume form. 
Indeed,
\[
\omega = \text{det} (\omega^{\ell}_i)_{\ell, i=1}^{d-1} \ud x^1 \wedge \cdots  \wedge \ud x^{d-1},
\]
and
\[
\text{det } (\omega^{\ell}_i)_{\ell, i=1}^{d-1} = \frac{\text{det} Q_{d-1}}{\sqrt{\lambda_1 \cdots \lambda_{d-1}}} = \sqrt{\frac{d}{\lambda_1 \cdots \lambda_{d-1}}} = \sqrt{|g|}. 
\]
Thus
\[
\text{vol}_g = \omega^1 \wedge \cdots \wedge \omega^{d-1} = \sqrt{|g|} \ud x^1 \wedge \cdots \wedge \ud x^{d-1} = \sqrt{|g|} \ud x^1 \cdots \ud x^{d-1},
\]
where $\ud x^1 \cdots \ud x^{d-1}$ is the Lebesgue measure on $\Omega_{d-1}$.
In the coordinate representation $(x^1, \cdots, x^{d-1})$, we have for the height function for the simplex $\sS$: $h(x^1, \cdots, x^{d-1}) = 1 - x^1 - \cdots - x^{d-1}$ and
\[
\ud \mathcal{H}^{d-1}(x) = \sqrt{1 + |\nabla h|^2} \ud x^1 \cdots \ud x^{d-1} = \sqrt{d} \ud x^1 \cdots \ud x^{d-1}.
\]
Thus, the volume form can be recast in terms of the Hausdorff surface measure for   $\sS$
\begin{equation}\label{volH}
 \text{vol}_g = \frac{\sqrt{|g|}}{\sqrt{d}} \ud \mathcal{H}^{d-1}(x) = \frac{1}{\sqrt{\lambda_1 \cdots \lambda_{d-1}}} \ud \mathcal{H}^{d-1}(x).   
\end{equation}
We complete the proof of (iii).
\end{proof}


\subsection{Laplace-Beltrami and Dirichlet form on the probability simplex}\label{sec2.5}
Before introducing an appropriate diffusion process on the probability simplex, we will first introduce an extrinsic representation for the Laplace-Beltrami operator and Dirichlet form on the probability simplex.

We first give an extrinsic representation of the grad operator and div operator. Recall grad operator for any $f\in C^1(\sS)$,
\[
 \grad_g f = \sum_{i,j=1}^{d-1} g^{ij} \frac{\partial f}{\partial x^j} \frac{\partial}{\partial x^i}.
\]
Then using the natural identical embedding $i: \sS \to \mathbb{R}^d$,
\begin{equation}\label{grad-f}
\begin{aligned}
i_{*} (\grad_g f) & = \sum_{i,j=1}^{d-1} g^{ij} \frac{\partial f}{\partial x^j} i_{*} \frac{\partial}{\partial x^i} 
= \sum_{i,j=1}^{d-1} \left( K_{ij} (e_j-e_d) \cdot \nabla f \right) (e_i-e_d) \\
& = \sum_{i=1}^{d} \left( \sum_{j=1}^{d} K_{ij} (e_j-e_d) \cdot \nabla f \right) (e_i-e_d) 
= K \nabla f. 
\end{aligned}
\end{equation}
And for $u \in T_x \sS$,
\begin{equation}\label{div-f}
\ddiv_g u = \frac{1}{\sqrt{|g|}} \sum_{j=1}^{d-1} \frac{\partial}{\partial x^j} \left( \sqrt{|g|} u^j \right) = \frac{1}{\sqrt{|g|}} \nabla \cdot ( \sqrt{|g|} i_{*} u).
\end{equation}

\subsubsection{Laplace-Beltrami Operator}
We also have the following extrinsic representation of the Laplace-Beltrami operator and Dirichlet energy.

\begin{lem}\label{lem_ex}
For any $f \in C^2(\sS)$, the intrinsic Laplace-Beltrami operator can be represented extrinsically as
\begin{equation}\label{LB}
\Delta_g f = \frac{1}{\sqrt{|g|}} \left[ \nabla \cdot \left( \sqrt{|g|} K \nabla f \right) \right],
\end{equation}
and the Dirichlet energy density can be represented extrinsically as
\begin{align}\label{geom}
g(\grad_g f, \grad_g f) = \langle \nabla f, K \nabla f \rangle_{\mathbb{R}^d}.
\end{align}
\end{lem}

\begin{proof}
For any $f \in C^\infty(\sS)$, recall $F$ in \eqref{Fnot}.

First, the Laplace-Beltrami operator is given by
\[
\Delta_g f = \frac{1}{\sqrt{|g|}} \sum_{i,j=1}^{d-1} \frac{\partial}{\partial x^i} \left( \sqrt{|g|} g^{ij} \frac{\partial F}{\partial x^j} \right)
= \frac{1}{\sqrt{|g|}} \left[ \sum_{i,j=1}^{d-1} (\partial_i - \partial_d) \left( \sqrt{|g|} K_{ij} (\partial_j - \partial_d) f \right) \right].
\]
Now we observe that
\begin{align*}
\sum_{i,j=1}^{d-1} (\partial_i - \partial_d) \sqrt{|g|} K_{ij} (\partial_j - \partial_d) f & = \frac{1}{\sqrt{|g|}} \sum_{i,j=1}^{d} (\partial_i - \partial_d) \left( \sqrt{|g|} K_{ij} (\partial_j - \partial_d) f \right) \\
& = \frac{1}{\sqrt{|g|}} \sum_{i,j=1}^{d} \partial_i \left( \sqrt{|g|} K_{ij} \partial_j \right) f = \frac{1}{\sqrt{|g|}} \nabla \cdot \left( \sqrt{|g|} K \nabla f \right).
\end{align*}
Here we used the row sum zero and symmetric properties of $K$. Thus we have the extrinsic representation \eqref{LB} of the Laplace-Beltrami operator.

Second, in the coordinates $(x^1, \cdots, x^{d-1})$, the Dirichlet energy $g(\grad_g f, \grad_g f)$ 
has the following representation
\begin{equation}
\begin{aligned}
g(\text{grad}_g f, \text{grad}_g f) = \sum_{i,j=1}^{d-1} g^{ij} \partial_i F \partial_j F. 
\end{aligned}
\end{equation}
Using an extrinsic representation, we have
\begin{align*}
g(\grad_g f, \grad_g f) &= \sum_{i,j=1}^{d-1} K_{ij} (e_i - e_d) \cdot \nabla f (e_j - e_d) \cdot \nabla f \\
&= \sum_{i,j=1}^{d} K_{ij} (e_i - e_d) \cdot \nabla f (e_j - e_d) \cdot \nabla f  
= \langle \nabla f, K \nabla f \rangle_{\mathbb{R}^d},
\end{align*}
where in the last equality we used the row sum zero property of $K$ again.
\end{proof}

\subsubsection{Dirichlet Form}
We now define the Dirichlet form for a function $f({x})$ on $\sS$. 

Recall the coordinates $(x^1, \cdots, x^{d-1})$ in domain $\Omega_{d-1}$ and the extrinsic representation of the Dirichlet energy density in Lemma \ref{lem_ex}. From \eqref{geom}, we have
\begin{align*}
& \int_{\Omega_{d-1}} g(\grad_g f, \grad_g f ) \, \sqrt{|g|} \ud x^1 \cdots \ud x^{d-1} \\
= & \sum_{i,j=1}^{d-1} \int_{\Omega_{d-1}} \partial_i F \sqrt{|g|} g^{ij} \partial_j F \ud x^1 \cdots \ud x^{d-1} \\
= & -\int_{\Omega_{d-1}} (f \Delta_g f) \circ \varphi^{-1} \sqrt{|g|} \ud x^1 \cdots \ud x^{d-1} + \sum_{i,j=1}^{d-1} \int_{\partial \Omega_{d-1}} F n_i K_{ij} \partial_j F \, \sqrt{|g|} \ud s,
\end{align*}
where $n = (n_1, \cdots, n_{d-1})$ is the unit outer normal of $\partial \Omega_{d-1}$.

We extend $n$ to $\mathbb{R}^d$ as $\vec{n} = (n_1, \cdots, n_{d-1}, 0)$ and use the row sum zero property of $K$, we have
\[
\sum_{i,j=1}^{d-1}     n_i K_{ij} \partial_j F = (\vec{n}^\intercal K \nabla f) \circ \varphi^{-1}. 
\]
The boundary term vanishes if one imposes the no-flux boundary condition
$$
(\vec{n}^\intercal K \nabla f) \circ \varphi^{-1} \, \sqrt{|g|} \Big|_{\partial \Omega_{d-1}} = 0.
$$
In terms of integration on the probability simplex, 
\begin{align*}
\int_{\sP(V)} g(\grad_g f, \grad_g f) \, \vol_g = -\int_{\sP(V)} f \Delta_g f \, \vol_g + \text{B.C.}
\end{align*}
Here $\Delta_g$ is the Laplace-Beltrami operator on the simplex set. 
In the two-point example in Section \ref{sec4}, we can see that the boundary term becomes zero.



\section{Fluctuations and Diffusion Approximation}\label{sec3}

After preparing the Riemannian structure on probability simplex in the last section, we now focus on the diffusion approximation for the process $(\cv_t)_{t\ge0}$, which leads to a $\sP(V)$-valued diffusion process. Here $V$ is a finite state space, and thus the process takes value on the probability simplex $\sS$. We identified the probability density as a vector on the simplex. The diffusion is constructed from a degenerate Brownian motion, motivated by the extrinsic representation of gradient and divergence on $\sS$. In Section \ref{sec3-FP}, we will first introduce the desired Fokker-Planck equation on $\sS$, which possesses a Gibbs measure on $\sS$. Then in Section \ref{sec3.4}, corresponding to the Fokker-Planck equation, we construct Langevin dynamics on $\sS$ through an appropriate degenerate Brownian motion. Finally, we discuss the relation between the generator and the weighted Dirichlet form on weighted $L^2(\sS, \pi_g)$.

\subsection{Diffusion approximation motivated by the Fokker-Planck equation on $\sS$}\label{sec3-FP}

Recall the $\sS_h$-valued process $(\cv_t)_{t\ge0}$ in \eqref{Csde}. To study the fluctuations at different levels, we first describe its law.
Denote $\vxv \in \sS_h$ as the discrete state variable on the probability simplex lattice with grid size $h$, indexed as
$$\ell=(\ell_1, \cdots, \ell_d)\in \mathbb{N}^d, \quad |\ell|:= \ell_1+\ell_2+\cdots+ \ell_d, \quad \vxv=h(\ell_1, \cdots, \ell_d).$$
Denote the time marginal law of $(\cv_t)_{t\ge0}$ as $$p_{\V}(t,\vxv) = \mathbb{E}(\mathbbm{1}_{\vxv}(\cv_t)) \in \sP(\sS_h)=\sP(\sP_h(V)) \subset \sP(h\mathbb{N}^d).$$ 
Then the law satisfies the Kolmogorov forward equation (chemical master equation)
\begin{equation}\label{rp_eq}
\begin{aligned}
\frac{\ud}{\ud t} p_{\V}(t,\vxv) =& \frac{1}{h}\sum_{(i,j)\in E} \left( Q_{ji}(x^h_{\ell_j}+ h) p_{\V}(\vxv- \vec{\nu}_{ji} h,t) - Q_{ij} x^h_{\ell_i}  p_{\V}(\vxv,t)\right) 
\end{aligned}
\end{equation}
for $\vxv \in \sS_h.$ 

To characterize the large fluctuations with exponential asymptotic behavior, we make a change of variable
\begin{equation}\label{1.3}
p_{\V}(t,\vxv) =  e^{-\frac{\psi_{\V}(t,\vxv)}{h}}, \quad p_{\V}(0,\vxv)=p_0(\vxv).
\end{equation}
Then the equation for $\psi_{\V}(t,\vxv)$ satisfies for $\vxv \in \sS_h,$
\begin{equation}\label{upwind00}
\begin{aligned}
\partial_t \psi_{\text{h}}(\vxv) + \sum_{(i,j)\in E } \left( Q_{ji}(x^h_{\ell_j}+ h) e^{\left(\frac{\psi_{\V}(\vxv) - \psi_{\V}(\vxv- \vec{\nu}_{ji} h)}{h}\right)} - Q_{ij} x^h_{\ell_i}  \right)=0.
\end{aligned}
\end{equation}
As $h\to 0$ and $\vx = \lim_{h\to 0} \vxv$,  
the Taylor expansion formally gives the continuous Hamilton–Jacobi equation (HJE) on the continuous probability simplex $\vx\in \sP(V)$,
\begin{equation}\label{HJE}
\partial_t \psi(\vx) + \sum_{(i,j)\in E} \left( Q_{ji} x_j e^{ \partial_i \psi(\vx) -\partial_j \psi(\vx) } - Q_{ij} x_i \right)=0.  
\end{equation}
One can directly check that the relative entropy
$\psi^{ss}(\vx) = \mathrm{D}_{\mathrm{KL}}({x} || \vx^s) := \sum_i \left( x_i \ln \frac{x_i}{x^{s}_i} - x_i + x^s_i \right)$
is a stationary solution to the HJE in the detailed balance case \eqref{DB}. The HJE provides an estimate for the large derivation rate function \cite{GL2023}.

On the other hand, to characterize the small fluctuations, one can construct a diffusion approximation for the process $\cv$ via several methods. One method is to study the law of the process $\cv$ and the approximation for the law $p_h$. A well-known procedure in this category is the Kramers–Moyal approximation for the Kolmogorov forward equation.
Another method is to study a quadratic approximation for the Hamiltonian in HJE \eqref{HJE}, which gives an equivalent diffusion approximation; see a comparison in \cite[Section 5.2]{GL2022}. At the process level, another method is to use the central limit approximation for the  random time-changed Poisson process to obtain a random time-changed Brownian motion \cite{EthierKurtz, Anderson}.

Below, we focus on another diffusion approximation method which incorporates the gradient flow structure \eqref{g_d} and the Riemannian metric induced by the Onsager matrix $K$.

We will construct a $\sS$-valued diffusion process $(X_t)_{t\ge0}$ and express its density $\rho(t,x)$ relative to the volume measure as below. For any test function $f(x)$, 
\[
\mathbb{E}(f(X_t)) = \int_{\sS} f(x) \rho(t,x) \vol_g.
\]
Recall the extrinsic representation of gradient and divergence on $\sS$ in \eqref{grad-f} and \eqref{div-f}. We aim to have the density $\rho$ satisfy the following Fokker-Planck equation on $\sS$
\begin{equation}\label{FP-rho}
\partial_t \rho = \frac{1}{\sqrt{|g|}} \nabla \cdot \left( \rho \sqrt{|g|} K \nabla \psi^{ss} \right) + h \frac{1}{\sqrt{|g|}} \nabla \cdot \left( \sqrt{|g|} K \nabla \rho \right).   
\end{equation}
Here, the drift term is given by the mean field limit ODE \eqref{JKO0} in the Onsager's gradient flow form \eqref{g_d}, and the diffusion term is given by the extrinsic representation of the Laplace-Beltrami operator \eqref{LB}. $h=\frac{1}{N}$ is the small parameter representing  fluctuations of the  stochastic process in the large volume.

A convenient way to construct the process is to use density $p=\rho \sqrt{|g|}$ relative to the surface Hausdorff measure so that the Itô calculus can be directly performed in the Euclidean space. We first prove the following lemma.

\begin{lem}
For the same process $(X_t)_{t \ge 0}$, if we represent the density using $p(t,x)$, 
\[
\mathbb{E}(f(X_t)) = \int_{\sS} f(x) p(t,x) \frac{1}{\sqrt{d}} \, \ud \mathcal{H}^{d-1}(x),
\]
then $p(t,\vx) = \rho(t,x) \sqrt{|g(x)|}$, $\vx \in \sS$, and the Fokker-Planck equation in terms of $p(t,\vx)$ is
\begin{equation}\label{FP-p}
\partial_t p = h \nabla \cdot \left( e^{-\frac{\psi^{ss}}{h}} \sqrt{|g|} K \nabla \left( \frac{p}{e^{-\frac{\psi^{ss}}{h}} \sqrt{|g|}} \right) \right).   
\end{equation}
\end{lem}

\begin{proof}
First, from \eqref{volH} and $p=\rho \sqrt{|g|}$, the density $p$ of the process $(X_t)_{t\ge 0}$ with respect to the Hausdorff measure satisfies
\begin{align*}
\mathbb{E}(f(X_t)) & = \int_{\sS} f(x) \rho(t,x) \vol_g
= \int_{\sS} f(x) p(t,x) \frac{1}{\sqrt{d}} \ud \mathcal{H}^{d-1}(x) \\
\end{align*}
Second, from $p=\rho \sqrt{|g|}$, the Fokker-Planck equation \eqref{FP-rho} becomes
\begin{align*}   
\partial_t (\rho \sqrt{|g|}) = \nabla \cdot \left( \rho \sqrt{|g|} K \nabla \psi^{ss} \right) + h \nabla \cdot \left( \rho \sqrt{|g|} K \nabla \log \rho \right)
= h \nabla \cdot \left( e^{-\frac{\psi^{ss}}{h}} \sqrt{|g|} K \nabla \left( e^{\frac{\psi^{ss}}{h}} \rho \right) \right),
\end{align*}
which, in terms of $p$, becomes \eqref{FP-p}.
\end{proof}

For the above Fokker-Planck equation in terms of $p$, the invariant measure is the Gibbs measure 
\[
\pi_g = e^{-\frac{\psi^{ss}}{h}} \sqrt{|g|} = \exp\left(-\frac{\psi^{ss} - \tfrac{1}{2} h \log |g|}{h}\right)
\]
with the potential $\psi^{ss}$ and a volume element $\sqrt{|g|}$. We incorporate the volume element contribution into the potential $\psi^{ss}$ as follows 
\begin{equation}\label{psihh}
\hat{\psi}^{ss} = \psi^{ss} - \tfrac{1}{2} h \log |g|.
\end{equation}
We will derive the diffusion approximation in terms of  $\hat{\psi}^{ss}$, which includes the volume element contribution.
Thus, we have
\begin{equation}\label{FPK}
\partial_t p = \nabla \cdot \left( p K(\vx) \nabla \hat{\psi}^{ss} + h K \nabla p \right) = h \nabla \cdot \left( e^{-\frac{\hat{\psi}^{ss}}{h}} K \nabla \left( p e^{\frac{\hat{\psi}^{ss}}{h}} \right) \right).
\end{equation}
With the Fokker-Planck equation in terms of $p$, we will construct the Langevin dynamics on $\sS$ in the next subsection.

\subsection{Construction of Langevin dynamics on $\sS$ for the diffusion approximation}\label{sec3.4}

Let $K$ be the Onsager response matrix, which is non-negative and symmetric, with the sum of its rows equal to zero. Recall the $(d-1)$ orthonormal eigenvectors $u^\ell(\vx)\in\mathbb{R}^d$ of matrix $K$ with positive eigenvalues as 
\begin{equation*} 
K(\vx)u^\ell(\vx)=\lambda_\ell(\vx)u^\ell(\vx)  \quad \textrm{for $\ell=1,2,\cdots, d-1$.}
\end{equation*}
Then $K$ can be decomposed using a $d\times (d-1)$ matrix $\sigma$,
\[
K= \sigma \sigma^\intercal, \quad \sigma = ( \sqrt{\lambda_1} u^1|\sqrt{\lambda_2}u^2|\cdots|\sqrt{\lambda_{d-1}}u^{d-1}) 
\]
with $u^\ell$ satisfying the mean zero condition $\sum_{i=1}^d  u^\ell_i =0$.

We use the backward it\^o integral to construct a $\sS$-valued drift-diffusion process
\begin{equation}\label{df11}
\ud  {X}_t  =    -K({X}_t ) {\nabla \hat{\psi}^{ss}({X}_t )}   \ud t + \sqrt{2h} \sigma({X}_t ) \, \widehat{\ud} B_t,
\end{equation}
where $\widehat{\ud} B_t$ denotes the multiplicative noise in the backward It\^o integral sense \cite{kunitha1982backward}. In the standard forward It\^o integral sense, this reads
\begin{equation}\label{528}
\ud  {X}_t   = -  K({X}_t ) {\nabla \hat{\psi}^{ss}({X}_t )} \ud t + h\nabla \cdot K({X}_t ) \ud t + \sqrt{2 h}\sigma({X}_t)\ud B_t,
\end{equation}
where $B_t$ is a Brownian motion in $\mathbb{R}^{d-1}$.

Using the symmetry and row sum zero property of $K_{ij}$, we can recast the drift in antisymmetric form. The above SDE, written component-wise, is
\begin{equation}\label{SFPKnew1}
\begin{split}
\ud X_i(t)=& \sum_{j=1, j\ne i}^d \left(-K_{ij} (\partial_j-\partial_i)\hat{\psi}^{ss} + h(\partial_j-\partial_i) K_{ij} \right) \ud t + \sqrt{2h}\sum_{\ell=1}^{d-1}\sqrt{\lambda_\ell} u^\ell_i \ud B_{\ell},
\end{split}
\end{equation}
where $B_{\ell}$ for $\ell=1, \cdots, d-1$ are independent Brownian motions. Using the Itô calculus, the Fokker-Planck equation for \eqref{528} is exactly the same as \eqref{FPK}.

Due to the fact that $\sum_{i=1}^d u^\ell_i = 0$ for $\ell=1, \cdots, d-1$, we know that 
$$\sum_{i=1}^d X_i(t) = \sum_{i=1}^d X_i(0) = 1$$ 
is conserved. However, the Brownian motion in \eqref{SFPKnew1} cannot guarantee the positivity of $X_t$. This is one of the central questions in probability.
An additional boundary condition is needed to ensure that the SDE \eqref{SFPKnew1} is a $\sP(V)$-valued diffusion process. One method is to use a reflecting boundary condition; see \cite{McKean} for a one-dimensional construction and \cite{Tanaka} for the case of a convex domain in multiple dimensions.
In Section \ref{sec4.1}, for the two-point case, i.e., $d=2$, we discuss the well-posedness in terms of the Fokker-Planck equation with a zero-flux boundary condition.

We also remark that there is another SDE constructed via edge-Brownian motion, which will lead to the same Fokker-Planck equation \eqref{FPK}. Let $B_{ij}$ for $i<j$ be independent Brownian motions on the edge $(i,j)$. For $j<i$, set $B_{ij}=-B_{ji}$. Then we introduce
\begin{equation}\label{SFPKnew2}
\begin{split}
\ud X_i(t)=& \sum_{j=1, j\ne i}^d \left(-K_{ij} (\partial_j-\partial_i)\hat{\psi}^{ss} + h(\partial_j-\partial_i) K_{ij} \right) \ud t + \sqrt{2h}\sum_{j=1, j\neq i}^d\sqrt{|K_{ij}|}\ud B_{ij},
\end{split}
\end{equation}
where $|\cdot|$ denotes the absolute value function. We refer to \cite{Yau}, where a non-degenerate edge Brownian motion was first proposed for a periodic lattice domain.

Next, we will prove that the Fokker-Planck equation of SDE \eqref{SFPKnew2} is exactly \eqref{FPK}. 
\begin{lem}
The Fokker-Planck equation corresponding to SDE \eqref{SFPKnew2} is \eqref{FPK}.
\end{lem} 
\begin{proof}
We first split the Brownian motion term into two parts
\begin{equation} 
\begin{split}
\ud X_i(t) =& \sum_{j=1, j\ne i}^d \left(-K_{ij} (\partial_j-\partial_i)\hat{\psi}^{ss} + h(\partial_j-\partial_i) K_{ij} \right) \ud t \\
&-\sqrt{2h} \sum_{j=1}^{i-1}\sqrt{|K_{ij}|}\ud B_{ji}
+\sqrt{2h} \sum_{j=i+1}^{d}\sqrt{|K_{ij}|}\ud B_{ij}.
\end{split}
\end{equation}

Then, by It\^o's formula, we compute the drift and diffusion terms.
For any test function $f \in C^2_b(\mathbb{R}^d)$, from  It\^o's formula
\begin{align*}
\ud f(X)=\sum_{n=1}^d \partial_n f \ud X_n + \frac{1}{2} \sum_{n,m=1}^d \partial_{nm} f \ud X_n \ud X_m.
\end{align*}
The quadratic term can be simplified as follows. For $n < m$, we have
\begin{align*}
\ud X_n \ud X_m = -2h |K_{nm}|\ud t = 2h K_{nm} \ud t.
\end{align*}
Similarly, for $n > m$, we have
\begin{align*}
\ud X_n \ud X_m = -2h |K_{nm}| \ud t = 2h K_{nm} \ud t.
\end{align*}
For $n = m$, we have
\begin{align*}
(\ud X_n)^2 = 2h \sum_{\ell \ne n} |K_{n\ell}| \ud t = 2h K_{nn} \ud t.
\end{align*}
Thus,
\begin{align*}
\frac{1}{2} \sum_{n,m=1}^d \partial_{nm}f  \ud X_n \ud X_m &= \frac{1}{2} \sum_{m, n; n \ne m} \partial_{nm} f  \ud X_n \ud X_m + \frac{1}{2} \sum_{n=1}^d \partial_{nn} f (\ud X_n)^2 \\
&= h \sum_{m, n; n \ne m} \partial_{nm} f K_{nm} \ud t + h \sum_{n=1}^d \partial_{nn}f  K_{nn} \ud t \\
&= h \sum_{n, m=1}^d \partial_{nm} f  K_{nm} \ud t. 
\end{align*}
Combine this with
\begin{align*}
\sum_{n=1}^d \partial_n f  \ud X_n = \sum_{n=1}^d \partial_n f  \left(- \sum_{\ell=1}^d K_{n\ell} \partial_\ell \hat{\psi}^{ss} + h\partial_\ell K_{n\ell} \right) \ud t + \mathcal{M}_t,
\end{align*}
where $\mathcal{M}_t$ is a martingale. Taking expectations and integrating by parts, we finish the proof.
\end{proof}

\subsection{Generator and weighted Dirichlet form}

The generator for the above diffusion process $(X_t)_{t\ge0}$ is given by
\[
L(f) = h e^{\frac{\psi^{ss}}{h}} \ddiv_g \left(e^{\frac{-\psi^{ss}}{h}} \grad_g f \right) 
= - g(\grad_g \psi^{ss}, \grad_g f) + h \ddiv_g \grad_g f.
\]
\textsc{Von Renesse and Sturm} introduced  a Wasserstein Dirichlet form \cite{WD} such that the generator for the diffusion process is the first variation of the Dirichlet form in a weighted $L^2$ space. In the current context, the weighted $L^2$ space is
\[
L^2(\sS, e^{-\frac{\psi^{ss}}{h}}\sqrt{|g|})= L^2(\sS, e^{-\frac{\hat{\psi}^{ss}}{h}} \mathcal{H}^{d-1}),
\]
and the Dirichlet form is given by
\[
\mathcal{E}(f) = \frac{h}{2} \int_{\sS} g(\grad_g f, \grad_g f) e^{-\frac{\psi^{ss}}{h}} \vol_g.
\]
The first variation of $\mathcal{E}(u)$ with respect to any compact support perturbation function $\tilde f$ is
\begin{align*}
\frac{d}{d\varepsilon}\big|_{\varepsilon=0} \mathcal{E}(f + \varepsilon \tilde f) &= h \int_{\sS} g(\grad_g f, \grad_g \tilde f) e^{-\frac{\psi^{ss}}{h}} \vol_g \\
&= h \int_{\sS} g(e^{-\frac{\psi^{ss}}{h}} \grad_g f, \grad_g \tilde f) \vol_g \\
&= -h \int_{\sS} \ddiv_g \left(e^{-\frac{\psi^{ss}}{h}} \grad_g f \right) \tilde f \vol_g \\
&= -h \int_{\sS} e^{\frac{\psi^{ss}}{h}} \ddiv_g \left(e^{-\frac{\psi^{ss}}{h}} \grad_g f \right) \tilde f e^{-\frac{\psi^{ss}}{h}} \vol_g.
\end{align*}
This recovers the generator operator
\[
L(f) = h e^{\frac{\psi^{ss}}{h}} \ddiv_g \left(e^{\frac{-\psi^{ss}}{h}} \grad_g f \right) 
= - g(\grad_g \psi^{ss}, \grad_g f) + h \ddiv_g \grad_g f.
\]
In the extrinsic representation, we have
\[
L(f) = \la K \nabla \psi^{ss}, \nabla f \ra_{\mathbb{R}^d} + h \frac{1}{\sqrt{|g|}} \nabla \cdot (\sqrt{|g|} K \nabla f).
\]

\begin{remark}\label{rem3.3}
Beyond the linear chemical reaction discussed above, for general chemical reactions, the reactions are indexed as $r=1:R$. For the $r$-th reaction, the reaction vectors are $\nu^r$, and the reaction rate $\Phi_r(\vx)$ is generally given by the law of mass action \cite{GL2022}. With detailed balance condition for general chemical reactions \cite[Proposition 4.4]{GL2022}, one will have the Onsager's matrix $K$  
\begin{equation}
K(x) = \sum_{r=1}^R \Lambda(\Phi_r^+(\vx), \Phi_r^-(\vx)) \vec{\nu}^r \otimes \vec{\nu}^r, \quad \text{ in detail } K_{ij}(x) = \sum_{r=1}^R \Lambda(\Phi_r^+(\vx), \Phi_r^-(\vx)) \nu^r_i \nu^r_j,
\end{equation}
where  
$\Lambda(x,y) = \frac{x-y}{\log x - \log y}.$
Thus, the SDE \eqref{SFPKnew2} with this matrix $K$ formulates as
\begin{equation*}
\begin{split}
\ud X_i(t)=& \sum_{j=1, j\ne i}^d \sum_{r=1}^R \left(-\Lambda(\Phi_r^+({X}_t), \Phi_r^-({X}_t)) \nu^r_i \nu^r_j (\partial_j-\partial_i) \hat{\psi}^{ss}({X}_t) + h (\partial_j-\partial_i) \Lambda(\Phi_r^+({X}_t), \Phi_r^-({X}_t)) \nu^r_i \nu^r_j \right) \ud t \\
&+\sqrt{2h} \sum_{j=1, j\neq i}^d \sqrt{|\sum_{r=1}^R \Lambda(\Phi_r^+({X}_t), \Phi_r^-({X}_t)) \nu^r_i \nu^r_j|} \ud B_{ij}.
\end{split}
\end{equation*}
We name the above SDE the {\em Wasserstein-Chemical Langevin equations}. We will leave its analysis and modeling for future work. 
\end{remark}


\section{Analysis of canonical Wasserstein diffusion on two-point simplex $\sP(\{1,2\})$}\label{sec4}
In this section, we study the existence, the conversion to the one-dimensional Wright-Fisher process, and the explicit solutions of the SDE \eqref{528} on the probability simplex in a two-point state space case.

\subsection{Two-point simplex example: a Langevin dynamics and Fokker-Planck equation on $\sP(\{1,2\})$}
Consider a two-point graph $I=\{1,2\}$, where the stationary solution is given by $x_1^{s}=x_2^{s}=\frac{1}{2}$. Hence $\omega_{12}=\omega_{21} = Q_{12} x_1^{s} >0$. Denote $k=-K_{12}=\theta_{12}({x}_1,x_2)\omega_{12}>0$. Then the Onsager response matrix is given by
\begin{equation*}
K=\begin{pmatrix}
k & -k\\
-k & k  
\end{pmatrix},
\end{equation*} 
with the positive eigenvalue being $\lambda_1=2k$. $K$ has a simple decomposition 
\[
K = \sigma \sigma^\intercal, \quad \sigma = \frac{\sqrt{\lambda_1}}{\sqrt 2} \begin{pmatrix} 1\\
-1  
\end{pmatrix}
= \sqrt{k}\begin{pmatrix}
1\\
-1  
\end{pmatrix}.
\]
Then SDE \eqref{528} reads
 \begin{equation}\label{eg1}
 \ud \begin{pmatrix}
X_1(t)\\
X_2(t) 
\end{pmatrix}
= - K \nabla \hat{\psi}^{ss}(X_1, X_2) \ud t + h \nabla\cdot K \ud t  + \sqrt{2h} \sigma \ud B_t.    
 \end{equation}
Clearly, we have the conservation law $X_1(t)+X_2(t)=X_1(0)+X_2(0)=1$. Denote $X_t=X_1(t)$ and $\hat{V}(x) = \hat{\psi}^{ss}(x, 1-x)$. We denote $\theta(x)=\theta_{12}(x, 1-x)$ and $\omega=\omega_{12}$. Then, the above equation is reduced to a one-dimensional SDE
\[
\ud X_t = - \omega \theta(X_t) \hat{V}'(X_t) \ud t + h \omega \theta'(X_t) \ud t + \sqrt{2h\omega \theta(X_t)} \ud B_t.
\]
Recall the formula for $\hat{\psi}^{ss}$ in \eqref{psihh} and the volume formula in Lemma \ref{lem2.2}, we have 
\[
\det g = \frac{2}{\lambda_1}=\frac{2}{2k} = \frac1{\omega \theta},
\]
and 
$$\hat V= V + \frac{h}{2} \log (\omega\theta), \quad \hat V'(x)= V'(x) + \frac{h}{2} \frac{\theta'(x)}{\theta(x)}.$$ Plugging this into \eqref{eg1}, we obtain
\begin{equation}\label{SDE2}
 \ud X_t = - \omega \theta(X_t) V'(X_t) \ud t + \frac{h}{2} \omega \theta'(X_t) \ud t + \sqrt{2h\omega \theta(X_t)} \ud B_t.   
\end{equation}

The Fokker-Planck equation of SDE \eqref{SDE2} satisfies
\begin{equation}\label{FP1}
\begin{split}
\partial_t p(t,x)
=&h\omega\partial_x\Big( p(t,x)\theta(x)\partial_x\log\frac{p(t,x)}{e^{-\frac{1}{h}V(x)}\theta(x)^{-\frac{1}{2}}}\Big)\\
=&h\omega\partial_x\Big(\theta(x)^{\frac{1}{2}}e^{-\frac{1}{h} V(x)}\partial_x[p(t,x)e^{\frac{1}{h}V(x)}\theta(x)^{\frac{1}{2}}]\Big), 
\end{split}
\end{equation}
with $p(0,x)=p_0(x)$ being a given initial distribution. We will impose the zero-flux boundary conditions in the Fokker-Planck equation \eqref{FP1}  
\begin{align}
&\lim_{x\to 0^+} \theta(x)^{\frac{1}{2}}e^{-\frac{1}{h} V(x)}\partial_x[p(t,x)e^{\frac{1}{h}V(x)}\theta(x)^{\frac{1}{2}}]=0, \label{BC11} \\
&\lim_{x\to 1^-}\theta(x)^{\frac{1}{2}}e^{-\frac{1}{h} V(x)}\partial_x[p(t,x)e^{\frac{1}{h}V(x)}\theta(x)^{\frac{1}{2}}] =0. \label{BC12}
\end{align}
This boundary condition results from Feller's boundary condition \cite{Feller51} for the degenerate Fokker-Planck equation. Later, Feller also proposed a boundary condition such that the trace of the generator vanishes at the boundary \cite{Feller54}, which was realized at the stochastic process level through reflecting Brownian motion \cite{McKean, Tanaka}. We will explain this further below after converting to the Wright-Fisher drift-diffusion process.
The stationary density of the SDE \eqref{SDE2} satisfies
\begin{equation*}
\pi(x)=\frac{1}{Z}e^{-\frac{V(x)}{h}}\theta(x)^{-\frac{1}{2}},
\end{equation*}
where we assume that $Z=\int_0^1 e^{-\frac{V(y)}{h}}\theta(y)^{-\frac{1}{2}} dy<+\infty$.

We also remark that to realize this zero-flux boundary condition, some modifications of Brownian motion in SDE \eqref{SDE2} at the boundary are needed. Typical examples include the killing process, reflection, or sticky Brownian motion \cite{Feller54, ItoMK, McKean, Tanaka}. Particularly, \cite{Anderson} proposed the reflecting boundary condition in the diffusion approximation of chemical reactions.  

\begin{example}[Canonical Wasserstein Diffusion \cite{WD, WL}]\label{ex1}
Consider $V(x)=0$, $\omega h =1$. Then SDE \eqref{SDE2} satisfies
\begin{equation}\label{SDE_X}
    \ud X_t=\frac{1}{2}\theta'(X_t)\ud t+\sqrt{2\theta(X_t)}\ud B_t.
\end{equation}
\end{example}
We will work on some detailed properties for \eqref{SDE_X} below.

\subsection{Conversion to Wright-Fisher process and Fokker-Planck equation with zero-flux boundary condition}\label{sec4.1}

We aim to transform the $(X_t)_{t\ge 0}$ SDE \eqref{SDE_X} into a Wright-Fisher diffusion process \cite{EM}
\begin{equation}
\ud Y_t= \gamma(\frac{1}{2}-Y_t) \ud t + \sqrt{2\gamma Y_t(1-Y_t)}\ud B_t,
\end{equation}
for some constant $\gamma>0$ to be determined. We note that the above SDE in terms of $(Y_t)_{t\ge 0}$ is a continuous Wright-Fisher model with mutations \cite[Formula (61)]{Burden}.

We first observe the following lemma for the transformation from $X_t$-SDE to $Y_t$-SDE.

\begin{lem}
The canonical Wasserstein diffusion process \eqref{SDE_X} can be converted to the Wright-Fisher drift-diffusion process after a change of variable. We find a strictly increasing bijective function $\psi: [0,1]\to[0,1]$, such that
\[
\sqrt{2\theta(x)} \psi'(x) = \sqrt{2\gamma \psi(x)(1-\psi(x))}.
\]
In other words,
\begin{equation}\label{ChangeY}
 \int_0^{\psi(x)}\frac{dz}{\sqrt{z(1-z)}}
=\sqrt{\gamma} \int_0^{x}\frac{dy}{\sqrt{\theta(y)}},
\end{equation}
where $\gamma$ is a constant satisfying
\[
\int_0^{1}\frac{dz}{\sqrt{z(1-z)}}
=\sqrt{\gamma} \int_0^{1}\frac{dy}{\sqrt{\theta(y)}}.
\]
Then SDE \eqref{SDE_X} can be recast as a SDE in terms of $Y_t=\psi(X_t)$
\begin{equation}
\ud Y_t = \gamma(\frac{1}{2}-Y_t) \ud t + \sqrt{2\gamma(Y_t(1-Y_t))}\ud B_t.
\end{equation}
\end{lem}
Let $\textbf{B}$ be the incomplete beta function defined as $\textbf{B}(x,a,b)=\int_0^x t^{a-1}(1-t)^{b-1}dt$. In terms of the beta function $\textbf{B}(\frac{1}{2},\frac{1}{2})$ and the incomplete beta function,
\begin{equation}\label{Bete-f}
\textbf{B}(\psi(x),\frac{1}{2}, \frac{1}{2})
=\sqrt{\gamma} \int_0^{x}\frac{dy}{\sqrt{\theta(y)}} \quad \text{ with } \sqrt{\gamma}= \frac{\textbf{B}(\frac{1}{2}, \frac{1}{2})}{\int_0^{1}\frac{dy}{\sqrt{\theta(y)}}}.
\end{equation}

\begin{proof}
The proof follows from a direct computation. Consider $Y_t=\psi(X_t)$. Then SDE \eqref{SDE_X} satisfies
\begin{equation*}
  \ud\psi(X_t)=\Big(\frac{1}{2} \psi'(X_t)\theta'(X_t)+\theta(X_t) \psi''(X_t)\Big) \ud t+\sqrt{2\theta(X_t)}\psi'(X_t) \ud B_t.   
\end{equation*}
To obtain a Wright-Fisher type degenerate diffusion, we need to solve
\begin{equation}\label{psiD}
 \sqrt{2\theta}\psi'=\sqrt{2\gamma\psi(1-\psi)},   
\end{equation} 
which implies \eqref{ChangeY} and
$$
\frac{1}{2} \theta' \psi' + \theta \psi'' =\gamma \left(\frac{1}{2}-\psi\right).
$$ 
We obtain SDE for $Y_t=\psi(X_t)$
\[
\ud Y_t = \gamma\left(\frac{1}{2}-Y_t\right) \ud t + \sqrt{2\gamma(Y_t(1-Y_t))}\ud B_t.
\]
\end{proof}

Next, to ensure the process is well-defined on $[0,1]$, we observe that the coefficients in $(Y_t)_{t\ge0}$-SDE satisfy the regular condition for the degenerate SDE \cite{Feller51}. Thus, modifications are needed whenever the Brownian motion hits a boundary. In terms of the Fokker-Planck equation, the zero-flux boundary condition needs to be imposed at $y=0,1$. Specifically,
the Fokker-Planck equation for the stochastic process $(Y_t)_{t\ge0}$ is
\[
\partial_t \tilde p + \gamma \partial_y((\frac{1}{2}-y)\tilde p) = \gamma \partial_{yy} (y(1-y)\tilde p).
\]
We impose zero-flux boundary conditions
\begin{equation}\label{0bc}
 \lim_{y\to 0^+} \left(\frac{1}{2}-y\right)\tilde p - \partial_{y} (y(1-y)\tilde p) = 
\lim_{y\to 0^+} \sqrt{y(1-y)} \partial_{y} (\tilde p\sqrt{y(1-y)}) =
0,   
\end{equation}
and
\begin{equation}\label{0bc1}
 \lim_{y\to 1^-} \left(\frac{1}{2}-y\right)\tilde p - \partial_{y} (y(1-y)\tilde p) = 
\lim_{y\to 1^-} \sqrt{y(1-y)} \partial_{y} (\tilde p\sqrt{y(1-y)}) =
0.   
\end{equation}
From the conservative form of the flux, one can directly verify that the invariant measure for $(Y_t)_{t\ge0}$ is
\[
\tilde{\pi}(y)=\frac{1}{\mathbf{B}(\frac{1}{2},\frac{1}{2})}y^{-\frac{1}{2}}(1-y)^{-\frac{1}{2}}.
\]
This is also a detailed balanced invariant measure since $\tilde{\pi}(y)$ has zero flux everywhere.

In terms of $(X_t)_{t\ge0}$ SDE, the Fokker-Planck equation is
\begin{equation}\label{FPttx}
\partial_t p = \partial_{xx}(p)- \frac{1}{2} \partial_x(\theta' p) = \pt_x \left[\sqrt{\theta(x)}\partial_x \left( \sqrt{\theta(x)} p(t,x)\right) \right].
\end{equation}
This can also be derived from the following two relations. From \eqref{psiD}, we have
\[
\sqrt{y(1-y)}\partial_y = \sqrt{\theta(x)}\partial_x;
\]
and 
\[
\tilde p(t,y)=\tilde p(t,\psi(x)) = \frac{p(t,x)}{\psi'(x)}.
\]
This implies
\[
\sqrt{y(1-y)} \tilde p = \sqrt{\theta(x)} p(t,x).
\]
The zero-flux boundary conditions \eqref{0bc}\eqref{0bc1} become
\[
\lim_{x\to 0^+} 
\sqrt{\theta(x)}\partial_x \left( \sqrt{\theta(x)} p(t,x)\right)= \lim_{x\to 0^+} 
-\frac{1}{2} \theta' \tilde p + \partial_x (\tilde p) = 0,
\]
and
\[
\lim_{x\to 1^-} 
\sqrt{\theta(x)}\partial_x \left( \sqrt{\theta(x)} p(t,x)\right)= \lim_{x\to 1^-} 
-\frac{1}{2} \theta' \tilde p + \partial_x (\tilde p) = 0.
\]
We can verify that the invariant measure for $(X_t)_{t\geq 0}$ is
\[
\pi(x) = \frac{1}{\mathbf{B}(\frac{1}{2},\frac{1}{2})}\psi(x)^{-\frac{1}{2}}(1-\psi(x))^{-\frac{1}{2}}\psi'(x)=\frac{\sqrt{\gamma}}{\mathbf{B}(\frac{1}{2},\frac{1}{2})} \frac{1}{\sqrt{\theta(x)}}
=\frac{1}{Z\sqrt{\theta(x)}}.
\]
The invariant measure $\pi(x)$ also satisfies the corresponding zero-flux boundary condition.

For the example $\theta(x)=2\sqrt{x(1-x)}$, \eqref{Bete-f} becomes
$$
\textbf{B}(\psi(x),\frac{1}{2}, \frac{1}{2})=\frac{\textbf{B}(\frac{1}{2}, \frac{1}{2})}{\sqrt{2}\textbf{B}(\frac{3}{4}, \frac{3}{4})} \textbf{B}(x, \frac{3}{4},\frac{3}{4}).
$$
Then the invariant distribution for $(X_t)_{t\geq 0}$ can be explicitly represented as
\begin{equation*}
  \pi(x)=\frac{1}{\mathbf{B}(\frac{3}{4},\frac{3}{4})}x^{-\frac{1}{4}}(1-x)^{-\frac{1}{4}}.
\end{equation*}

 
 \subsection{Fundamental solutions}\label{FSWFP}
We now derive the closed-form solution of the Fokker-Planck equation \eqref{FP1} with $V=0$ and $h\omega=1$ for the canonical Wasserstein diffusion \eqref{SDE_X}. 

Recall the Fokker-Planck equation \eqref{FPttx} for \eqref{SDE_X}
\begin{equation*}
\begin{split}
\partial_t p(t,x)
=\partial_x\Big(p(t,x)\theta(x)\partial_x\log\frac{p(t,x)}{\theta(x)^{-\frac{1}{2}}}\Big)
=\partial_x\Big(\sqrt{\theta(x)}\partial_x\left(p(t,x)\sqrt{\theta(x)}\right)\Big) 
\end{split}
\end{equation*}
with $p(0,x)=p_0(x)$ being a given initial distribution. Here, the zero-flux boundary condition reads
\begin{align}
&\lim_{x\to 0^+}\sqrt{\theta(x)}\partial_x\left(p(t,x)\sqrt{\theta(x)}\right)=0, \label{BC1} \quad 
&\lim_{x\to 1^-}\sqrt{\theta(x)}\partial_x\left(p(t,x)\sqrt{\theta(x)}\right)=0. 
\end{align}
The invariant distribution satisfies
\begin{equation*}
   \pi(x):= \frac{1}{Z}\theta(x)^{-\frac{1}{2}}, \quad Z:=\int_0^1\theta(r)^{-\frac{1}{2}}dr.
\end{equation*}

Denote the function
\[
h(t,x)= Z \theta(x)^{\frac{1}{2}} p(t,x).
\]
We have
\begin{equation*}
 \partial_t h(t,x)=\theta(x)^{\frac{1}{2}}  \partial_x\Big(\theta(x)^{\frac{1}{2}}\partial_x h(t,x)\Big), \quad h(0,x)=Z \theta(x)^{\frac{1}{2}} p_0(x).   
\end{equation*}
Let
\[
y(x)=\frac{1}{Z} \int_0^x\frac{1}{{\theta(r)^{\frac{1}{2}}}}dr,  \quad x\in[0,1].
\]
This is a bijective monotone map, such that $dy=\frac{dx}{Z\sqrt{\theta(x)}}$.  

The function $u(t,y):=h(t,x(y))$ satisfies the classical heat equation with imposed Neumann boundary condition 
\begin{equation}\label{heat}
    \partial_t u(t,y)=\frac{1}{Z^2}\partial^2_{yy}u(t,y),\quad    \partial_yu(t,0)=0, \,\, \partial_y u(t,1)=0,\quad \forall t>0. 
\end{equation}
One can verify this homogeneous Neumann boundary condition for $u$ is equivalent to the zero-flux boundary condition for $p(t,x)$ in the Fokker-Planck equation \eqref{BC1}
$$\partial_y u(t,y)\Big|_{y=0,1} =Z \theta(x)^{\frac{1}{2}}\partial_x h(t,x)\Big|_{x=0,1} = Z^2 \sqrt{\theta(x)}\partial_x\left(p(t,x)\sqrt{\theta(x)}\right)\Big|_{x=0,1}=0, \quad \forall t>0.$$
We can solve $u$ as
$$u(t,y)=\sum_{k=0}^{\infty}c_k e^{-(\frac{k\pi}{Z})^2 t} \cos( k\pi y ), \quad c_0 = \int_0^1 u_0(y) \ud y = 1, \,\,   c_k = 2\int_0^1 u_0(y) \cos( k\pi y ) \ud y, \, k>1.$$ 

There is an explicit solution for the discrete Wasserstein-2 distance between $p^0=(0,1)$ and $p^1=(x, 1-x)$ in a two-point space \cite{EM1, GLL}  
\[
W_2(p^0,p^1)=W_2((0,1),(x,1-x))= \int_0^x\frac{1}{{\theta(r)^{\frac{1}{2}}}}dr = y(x) Z.
\] 
Plugging this into $u(t,y)$, we obtain
\begin{equation*}
\begin{split}
    p(t,x)
=&\frac{\theta(x)^{-\frac{1}{2}}}{Z}\sum_{k=0}^{\infty}c_k e^{-(\frac{k\pi}{Z})^2 t} \cos( \frac{k\pi}{Z} W_2((0,1),(x,1-x))),
\end{split}
\end{equation*}
where
\begin{equation*}
\begin{split}
 c_0=1, \quad c_k 
=&2  \int_0^1 \rho_0(z) \cos( \frac{k\pi}{Z} W_2((0,1),(z,1-z))) \ud z, \quad k>0.
\end{split}
\end{equation*}

Denote $G$ as the Green function on an interval $[0,1]$
\begin{equation*}
  G(t,x,z)=  \frac{1}{Z}\theta(x)^{-\frac{1}{2}}+\frac{2}{Z}\theta(x)^{-\frac{1}{2}}\sum_{k=1}^\infty e^{-(\frac{k\pi}{L})^2 t}\cos( \frac{k\pi}{Z} W_2((0,1),(z,1-z)))\cdot \cos( \frac{k\pi}{Z} W_2((0,1),(x,1-x))), 
\end{equation*}
where $x,z\in [0,1]$. We call $G$ the {\em ``Wasserstein Green function''} on a two-point graph. 
Thus, 
\begin{equation*}
   p(t,x)=\int_0^1 G(t,x, z)\rho_0(z)dz, 
\end{equation*}
 
We remark that the beta function reformulates the Wasserstein distance on a graph in the Green function for the case $\theta(x)=2 \sqrt{x(1-x)}$. Notice
\begin{equation*}
W_2((0,1),(x,1-x))=\int_0^x\frac{1}{\sqrt{2}r^{\frac{1}{4}}(1-r)^{\frac{1}{4}}}dr=\frac{1}{\sqrt{2}}\mathbf{B}(x,\frac{3}{4},\frac{3}{4}), 
\end{equation*}
where $\textbf{B}$ is the incomplete beta function defined as $\textbf{B}(x,a,b)=\int_0^x t^{a-1}(1-t)^{b-1}dt$.

 
 \section{Discussions}
This paper studies the fluctuations of gradient flows on probability spaces. Here, the gradient flow can be regarded as the Kolmogorov forward equation of a Markov process, while the fluctuations can originate from different constructions. We construct fluctuations resulting from the diffusion approximation of the frequency process for the molecular numbers of a detailed balanced linear chemical reaction. This leads to Langevin dynamics on a probability simplex with a Riemannian structure suggested by Onsager's response matrix.

The Onsager response matrix in the gradient flow naturally induces a Riemannian metric on the probability simplex, which can also be interpreted as the Wasserstein-2 type metric on graphs. We also study the Dirichlet form in the $L^2$ space weighted by a Gibbs measure. The first variation of this Dirichlet form gives the generator of the Langevin dynamics. Detailed properties of the Fokker-Planck equations and the connection with the Wright-Fisher process were established on a two-point state space. 

A well-known fact is that the detailed-balanced Markov process is the gradient flow of free energy in the Wasserstein-2 metric on graphs \cite{EM1, MM, M, chow2012}. Extending this gradient flow structure with fluctuations, we further build a drift-diffusion process on discrete Wasserstein spaces, analogous to stochastic Fokker-Planck equations. Onsager's response matrix facilitates the construction of a $(d-1)$-dimensional degenerate Brownian motion on the probability simplex, whose generator corresponds to the Laplace-Beltrami operator on discrete Wasserstein space. In particular, there is degeneracy near the boundary of the simplex set. The well-definedness of the proposed SDEs was partially verified via the connection with the one-dimensional Wright-Fisher process for population genetics in a two-point state space. The general case remains unsolved. 

There are also several open questions in this field. For example, what are the general transition probabilities for Wasserstein drift-diffusion processes? What are finite-dimensional approximations of Wasserstein diffusions on general probability spaces, such as Dean-Kawasaki dynamics? From a modeling perspective, one can also explore applying Wasserstein diffusion on probability simplex to study   multi-allelic  models in population genetics \cite{Burden}.   Efficient numerical simulations of Wasserstein diffusion processes are also important future directions. 

\bibliographystyle{abbrv}

\end{document}